\documentclass[12pt,leqno]{amsart}
\usepackage{amsmath,amsfonts,amssymb,amsthm}
\usepackage{graphicx}
\graphicspath{ }
\usepackage{wrapfig}
\usepackage{accents}
\usepackage{caption}
\usepackage{subcaption}
\usepackage{calligra}

\numberwithin{figure}{section}

\theoremstyle{plain}
\newtheorem{thm}{Theorem}[section]

\theoremstyle{definition}
\newtheorem{defn}{Definition}[section]

\theoremstyle{remark}

\newtheorem*{note}{Note}
\usepackage{mathtools}
\title{Rectifying curves on a smooth surface immersed in the euclidean space}
\author[A. A. Shaikh and P. R. Ghosh]{Absos Ali Shaikh$^*$ and Pinaki Ranjan Ghosh}
\address{\noindent\newline  Department of Mathematics,\newline University of
Burdwan, Golapbag,\newline Burdwan-713104,\newline West Bengal, India}
\email{aask2003@yahoo.co.in, aashaikh@math.buruniv.ac.in}
\email{mailtopinaki94@gmail.com}
\begin{document}

\begin{abstract}
The main objective of the present paper is to investigate a sufficient condition for which a rectifying curve on a smooth surface remains invariant under isometry of surfaces, and also it is shown that under such an isometry the component of the position vector of a rectifying curve on a smooth surface along the normal to the surface is invariant.
\end{abstract}
\noindent\footnotetext{ $^*$ Corresponding author.\\
$\mathbf{2010}$\hspace{5pt}Mathematics\; Subject\; Classification: 53A15, 53A04, 53A05, 51M05. }
\maketitle
\section{Introduction}
In 2003, Bang-Yen Chen \cite{BYC03} introduced the notion of the rectifying curve in the Euclidean space $\mathbb{R}^3$ as a curve whose position vector lies in the rectifying plane and such a curve classified by an unit speed curve in an unit sphere $S^2$ and also obtained some of its characterization. For further properties of rectifying curves we refer the reader to see \cite{BYC05} and \cite{BYC18}. By motivating the above studies, the main goal of this paper is to investigate the nature of rectifying curves on a smooth surface $S$ under an isometry to another smooth surface $\bar{S}$. Then we obtain a sufficient condition for which a rectifying curve on $S$ remains invariant under isometry $F:S\rightarrow \bar{S}$. We also note that under isometry of $\mathbb{R}^3$, a rectifying curve on $\mathbb{R}^3$ is not necessarily transformed to a rectifying curve on $\mathbb{R}^3$. It is also shown that the component of the position vector of a rectifying curve on a smooth surface along the normal to the surface is invariant under the rectifying curve preserving isometry of surfaces.
\par
The structure of the paper is as   follows. Section 2 deals with the discussion of some rudimentary facts of Frenet-Serret equations and rectifying curves. Section 3 is devoted to the study of rectifying curves on a smooth surface and deduced the components of position vectors of such a curve along the normal to the surface. The last section is concerned with the main result (see Theorem \ref{thm1}, Theorem \ref{thm2}).
\section{Preliminaries}
In this section, we recall some rudimentary facts of rectifying curves, isometry of surfaces and first fundamental form (for details see, \cite{AP01}, \cite{MPDC76 }) which will be used throughout the paper.
\par
Let $\gamma(s):I\rightarrow \mathbb{R}^3$, where $I=(\alpha,\beta)\subset\mathbb{R}$, be a unit speed parametrized curve having at least fourth order continuous derivatives. Let the tangent vector of the curve $\gamma(s)$ be denoted by $\vec{t}$. We consider $\vec{t'}(s)\neq 0$, so that there is an unit normal vector $\vec{n}$ along $\vec{t'}(s)$ and also a positive function $k(s)$ such that $\vec{t'}(s)=k(s)\vec{n}(s)$, where $\vec{t'}$ denote the derivative with respect to the arc length parameter $s$. The binormal vector field is defined by $\vec{b}=\vec{t}\times \vec{n}$. There is another curvature function $\tau(s)$, called torsion, and is given by the equation $\vec{b'}(s)=\tau(s)\vec{n}(s)$. At each point on $\gamma(s)$, $\{\vec{t},\vec{n},\vec{b}\}$ forms an orthonormal frame. At every point of the curve $\gamma(s)$, the planes generating by $\{\vec{t},\vec{n}\}$,$\{\vec{n},\vec{b}\}$ and $\{\vec{b},\vec{t}\}$ are called osculating plane, normal plane and rectifying plane respectively. The quantity $ \|\vec{b'}(s)\|$ measures the rate of change of the neighboring osculating plane with the osculating plane at $s$. The Frenet-Serret equations are given by
\begin{eqnarray*}
\vec{t'}&=& k\vec{n},\\
\vec{n'}&=&-k\vec{t}+\tau \vec{b},\\
\vec{b'}&=& -\tau \vec{n}.
\end{eqnarray*}
A curve in $\mathbb{R}^3$ is called rectifying \cite{BYC03} if its position vector always lies in the rectifying plane of that curve. The position vector $\gamma(s)$ satisfies the equation
\begin{equation*}
\gamma(s)=\lambda(s)t(s)+\mu(s)b(s),
\end{equation*}
 for some functions $\lambda(s)$ and $\mu(s)$.
\par
Let $\gamma(t)=\phi(u(t),v(t))$, where $t \in (a,b)\subset\mathbb{R}$,  be a curve in a surface patch $\phi$. Then $\{\phi_u,\phi_v\}$ are linearly independent, and hence generates the tangent space $T_p\phi$ at a point $p\in\phi$. Thus we have
\begin{eqnarray*}
\|\dot{\gamma(t)}\|&=&(\phi_u\dot{u}+\phi_v\dot{v})\cdot(\phi_u\dot{u}+\phi_v\dot{v}),\\
&=&(\phi_u\cdot\phi_u)\dot{u}^2+2(\phi_u\cdot\phi_v)\dot{u}\dot{v}+(\phi_v\cdot\phi_v)\dot{v}^2,\\
&=&E\dot{u}^2+2F\dot{u}\dot{v}+G\dot{v}^2,
\end{eqnarray*}
where $\dot{\gamma}(t)$ denote the derivative with respect to the parameter $t$.
\par
A surface $S$ is said to be regular if, for each $p\in S$ there exists a neighborhood $V\subset\mathbb{R}^3$ and a map $\psi: U\rightarrow V\cap S$ of an open set $U\subset \mathbb{R}^2$ onto $V\cap S\subset \mathbb{R}^3$ such that $\psi$ is differentiable, homeomorphism and the differential $d\psi_q$ is one to one for all $q\in U$.
\begin{defn}
The first fundamental form of a regular surface $S$ at a point $p$ is a quadratic form $I_p:T_pS\rightarrow \mathbb{R}$ given by 
\begin{equation*}
I_p(\dot{\gamma}(t))=<\dot{\gamma}(t),\dot{\gamma}(t)>=\|\dot{\gamma}(t)\|.
\end{equation*}
\end{defn}
\begin{defn}
A diffeomorphism $F:S\rightarrow \bar{S}$, where $S$ and $\bar{S}$ are smooth surfaces in $\mathbb{R}^3$, is an isometry if $F$ takes a curve from $S$ to a curve of same length on $\bar{S}$.
\par
Isometry of $\mathbb{R}^3$ is uniquely described as an orthogonal transformation followed by a translation. If we rotate the rectifying curve $\gamma(s)$ by fixing a point $\gamma(s_0)$ then at $\gamma(s_0)$, the Frenet-Serret frame transforms into another frame. Hence at $\gamma(s_0)$ the corresponding rectifying plane transforms into another rectifying plane. But the position vector of the curve $\gamma(s)$ does not change before and after the rotation. Therefore, generally, rectifying curves are not invariant under the isometry of $\mathbb{R}^3$.
\end{defn}

\section{Rectifying curves  on smooth surfaces}
Let $\phi:U\rightarrow S$ be the coordinate chart for a smooth surface $S$ and the unit speed parametrized curve $\gamma(s):(\alpha,\beta)\rightarrow S$, where$(\alpha,\beta)\subset\mathbb{R}$ contained in the image of a surface patch $\phi$ in the atlas of $S$. Then $\gamma(s)$ is given by, 
\begin{eqnarray}\label{1}
\nonumber(\alpha,\beta) \rightarrow U,
\nonumber\text{\quad}s\rightarrowtail (u(s),v(s),\\
\gamma(s)=\phi(u(s),v(s)).
\end{eqnarray}
Differentiating $(\ref{1})$ with respect to $s$, we get
\begin{eqnarray}
\gamma'(s)&=&\phi_uu'+\phi_vv',\\
\text{i.e., }\nonumber \vec{t}(s)&=&\gamma'(s)=\phi_uu'+\phi_vv',\\
\nonumber
\text{hence, } \vec{t'}(s)&=& u''\phi_u+v''\phi_v+u'^2\phi_{uu}+2u'v'\phi_{uv}+v'^2\phi_{vv}.
\end{eqnarray}
If $k(s)$ is the curvature of $\gamma(s)$ and $\vec{N}$ is normal to $S$ then the normal $\vec{n}(s)$ is given by
\begin{eqnarray}
\nonumber
\vec{n}(s)&=&\frac{1}{k(s)}(u''\phi_u+v''\phi_v+u'^2\phi_{uu}+2u'v'\phi_{uv}+v'^2\phi_{vv}).\\
\nonumber
\vec{b}(s)&=& \vec{t}(s)\times \vec{n}(s)= \vec{t}(s)\times \frac{\vec{t'}(s)}{k(s)},\\
\nonumber
&=&\frac{1}{k(s)}\Big[(\phi_uu'+\phi_vv')\times(u''\phi_u+v''\phi_v+u'^2\phi_{uu}+2u'v'\phi_{uv}+v'^2\phi_{vv})\Big],\\
\nonumber
&=& \frac{1}{k(s)}\Big[u'v''\vec{N}+u'^3\phi_u\times \phi_{uu}+2u'^2v'\phi_u\times \phi_{uv}+u'v'^2\phi_u\times \phi_{vv}-u''v'\vec{N}+\\
\nonumber
&& u'^2v'\phi_v\times \phi_{uu}+2u'v'^2\phi_v\times \phi_{uv}+v'^3\phi_v\times \phi_{vv}\Big],\\
\nonumber
&=&\frac{1}{k(s)}\Big[\{u'v''-u''v'\}\vec{N}+u'^3\phi_u\times \phi_{uu}+2u'^2v'\phi_u\times \phi_{uv}+u'v'^2\phi_u\times \phi_{vv}\\
\nonumber
&&+u'^2v'\phi_v\times \phi_{uu}+2u'v'^2\phi_v\times \phi_{uv}+v'^3\phi_v\times \phi_{vv}\Big].
\end{eqnarray}
So, $\gamma(s)$ in $S$ will be rectifying curve if $\gamma(s)=\lambda(s)t(s)+\mu(s)b(s)$, for some functions $\lambda(s)$ and $\mu(s)$.
i.e.,
\begin{eqnarray}\label{2}
\nonumber
\gamma(s)&=&\lambda(s)(\phi_uu'+\phi_vv')+\frac{\mu(s)}{k(s)}\Big[\{u'v''-u''v'\}\vec{N}+u'^3\phi_u\times \phi_{uu}+2u'^2v'\phi_u\times \phi_{uv}+u'v'^2\phi_u\times \phi_{vv}\\
\nonumber
&&+u'^2v'\phi_v\times \phi_{uu}+2u'v'^2\phi_v\times \phi_{uv}+v'^3\phi_v\times \phi_{vv}\Big]
\end{eqnarray}
  for some functions $\lambda(s)$ and $\mu(s)$.
\par
Now we find component of the position vector of the curve $\gamma(s)$ along the normal $\vec{N}$ to the surface $S$ at a point $\gamma(s)$ and obtain
\begin{eqnarray}
\nonumber
\gamma(s)\cdot \vec{N}&=&\lambda(s)(\phi_uu'+\phi_vv')+\frac{\mu(s)}{k(s)}\Big[\{u'v''-u''v'\}\vec{N}+u'^3\phi_u\times \phi_{uu}+2u'^2v'\phi_u\times \phi_{uv}\\
\nonumber
&&+u'v'^2\phi_u\times \phi_{vv}+u'^2v'\phi_v\times \phi_{uu}+2u'v'^2\phi_v\times \phi_{uv}+v'^3\phi_v\times \phi_{vv}\Big]\cdot\vec{N},\\
\nonumber
&=&\frac{\mu(s)}{k(s)}\Big[(u'v''-u''v) (EG-F^2)+u'^3(\phi_u\times \phi_{uu})\cdot\vec{N}+2u'^2v'(\phi_u\times \phi_{uv})\cdot\vec{N}\\
\nonumber
&&+u'v'^2(\phi_u\times \phi_{vv})\cdot\vec{N}+u'^2v'(\phi_v\times \phi_{uu})\cdot\vec{N}+2u'v'^2(\phi_v\times \phi_{uv})\cdot\vec{N}+v'^3(\phi_v\times \phi_{vv})\cdot\vec{N}\Big],\\
&=&\frac{\mu(s)}{k(s)}\Big[(u'v''-u''v)(EG-F^2)+u'^3\{E(\phi_{uu}\cdot\phi_v)-F(\phi_{uu}\cdot\phi_u)\}+2u'^2v'\{E(\phi_{uv}\cdot\phi_v)\\
\nonumber
&&-F(\phi_{uv}\cdot\phi_u)\}+u'v'^2\{E(\phi_{vv}\cdot\phi_v)-F(\phi_{vv}\cdot\phi_u)\}+u'^2v'\{F(\phi_{uu}\cdot\phi_v)-G(\phi_{uu}\cdot\phi_u)\}\\
\nonumber
&&+2u'v'^2\{F(\phi_{uv}\cdot\phi_v)-G(\phi_{uv}\cdot\phi_u)\}+v'^3\{F(\phi_{vv}\cdot\phi_v)-G(\phi_{vv}\cdot\phi_u)\}.
\end{eqnarray}

\section{Main result}

In the following theorem we consider the expression $F_*(\gamma(s))$ as a product of a $3\times 3$ matrix $F_*$ and a $3\times 1$ matrix $\gamma(s)$.
\begin{thm}\label{thm1}
Let $F:S\rightarrow \bar{S}$ be an isometry, where $S$ and $\bar{S}$ are smooth surfaces and $\gamma(s)$ be a rectifying curve on $S$. Then $\bar{\gamma}(s)$ is a rectifying curve on $\bar{S}$ if
\begin{eqnarray}
\nonumber
\bar{\gamma}(s)-F_*(\gamma(s))&=&\frac{\mu(s)}{k(s)}\Big[u'^3\Big(F_*\phi_u\times \frac{\partial F_*}{\partial u}\phi_u\Big)+2u'^2v'\Big(F_*\phi_u\times \frac{\partial F_*}{\partial u}\phi_v\Big)+u'v'^2\Big(F_*\phi_u\times \frac{\partial F_*}{\partial v}\phi_v\Big)\\
&&+u'^2v'\Big(F_*\phi_v\times \frac{\partial F_*}{\partial u}\phi_u\Big)+2u'v'^2\Big(F_*\phi_v\times \frac{\partial F_*}{\partial u}\phi_v\Big)+v'^3\Big(F_*\phi_v\times \frac{\partial F_*}{\partial v}\phi_v\Big)\Big].
\end{eqnarray}
\end{thm}
\begin{proof}
Let $\phi$ and $\bar{\phi}$ be the coordinate charts for $S$ and $\bar{S}$ respectively, where
\begin{equation*}
\bar{\phi}=F\circ\phi.
\end{equation*} 
\par
The tangent plane at a point $p$ on $S$ is generated by two vectors $\phi_u$ and $\phi_v$. Since $F$ is an isometry between $S$ and $\bar{S}$, the differential map $F_*$ of $F$ is a $3\times 3$ orthogonal matrix. Therefore $F_*$ takes linearly independent vectors $\phi_u$ and $\phi_v$ of $T_pS$ to $\bar{\phi}_u$ and $\bar{\phi}_v$ of $T_{F(p)}S$. Also $\vec{N}$ and $\vec{\bar{N}}$ are normals to $S$ and $\bar{S}$ respectively.\\
\begin{eqnarray}
&&\bar{\phi}_u(u,v)=F_*\phi_u=F_*(\phi(u,v))\phi_u,\\
&&\bar{\phi}_v(u,v)=F_*\phi_v=F_*(\phi(u,v))\phi_v.
\end{eqnarray}
Again differentiating $(5)$ and $(6)$ partially with respect to both $u$ and $v$ respectively, we get
\begin{eqnarray}
\nonumber \bar{\phi}_{uu}&= \frac{\partial F_*}{\partial u}\phi_u+F_*\phi_{uu},\\
\bar{\phi}_{vv}&= \frac{\partial F_*}{\partial v}\phi_v+F_*\phi_{vv},\\
\nonumber \bar{\phi}_{uv}&= \frac{\partial F_*}{\partial u}\phi_v+F_*\phi_{uv},\\
\nonumber &= \frac{\partial F_*}{\partial v}\phi_u +F_*\phi_{uv}.\\\nonumber
 \end{eqnarray}

Now
\begin{equation}\label{**}
F_*\phi_u\times \frac{\partial F_*}{\partial u}\phi_u=F_*\phi_u\times\Big(\frac{\partial F_*}{\partial u}\phi_u+F*\phi_{uu}\Big)-F_*(\phi_u\times \phi_{uu})=\bar{\phi}_u\times \bar{\phi}_{uu}-F_*(\phi_u\times \phi_{uu}).
\end{equation}
Similarly 
\begin{eqnarray}\label{***}
\nonumber F_*\phi_u\times \frac{\partial F_*}{\partial u}\phi_v&=\bar{\phi}_u\times \bar{\phi}_{uv}-F_*(\phi_u\times \phi_{uv}),\\
\nonumber F_*\phi_u\times \frac{\partial F_*}{\partial v}\phi_v&=\bar{\phi}_u\times \bar{\phi}_{vv}-F_*(\phi_u\times \phi_{vv}),\\
F_*\phi_v\times \frac{\partial F_*}{\partial u}\phi_u&=\bar{\phi}_v\times \bar{\phi}_{uu}-F_*(\phi_v\times \phi_{uu}),\\
\nonumber F_*\phi_v\times \frac{\partial F_*}{\partial u}\phi_v&=\bar{\phi}_v\times \bar{\phi}_{uv}-F_*(\phi_v\times \phi_{uv}),\\
\nonumber F_*\phi_v\times \frac{\partial F_*}{\partial v}\phi_v&=\bar{\phi}_v\times \bar{\phi}_{vv}-F_*(\phi_v\times \phi_{vv}).
 \end{eqnarray}

In view of $(4)$, (\ref{**}) and $(\ref{***})$, we get
\begin{eqnarray}
\nonumber
\bar{\gamma}(s)&=&\lambda(s)(u'F_*\phi_u+v'F_*\phi_v)+\frac{\mu(s)}{k(s)}\Big[\{u'v''-u''v'\}F_*\vec{N}+u'^3F_*(\phi_u\times \phi_{uu})\\
\nonumber
&&+2u'^2v'F_*(\phi_u\times \phi_{uv})+u'v'^2F_*(\phi_u\times \phi_{vv})+u'^2v'F_*(\phi_v\times \phi_{uu})+2u'v'^2F_*(\phi_v\times \phi_{uv})\\
\nonumber
&&+v'^3F_*(\phi_v\times \phi_{vv})+u'^3\Big(F_*\phi_u\times \frac{\partial F_*}{\partial u}\phi_u\Big)+2u'^2v'\Big(F_*\phi_u\times \frac{\partial F_*}{\partial u}\phi_v\Big)\\
\nonumber
&&+u'v'^2\Big(F_*\phi_u\times \frac{\partial F_*}{\partial v}\phi_v\Big)+u'^2v'\Big(F_*\phi_v\times \frac{\partial F_*}{\partial u}\phi_u\Big)+2u'v'^2\Big(F_*\phi_v\times \frac{\partial F_*}{\partial u}\phi_v\Big)\\
\nonumber
&&+v'^3\Big(F_*\phi_v\times \frac{\partial F_*}{\partial v}\phi_v\Big)\Big].
\end{eqnarray}
Which can be written as
\begin{eqnarray*}
\bar{\gamma}(s)=\lambda(s)\Big(u'\bar{\phi}_u+v'\bar{\phi}_v\Big)+\frac{\mu(s)}{k(s)}\Big[\{u'v''-u''v'\}\vec{\bar{N}}+u'^3\bar{\phi}_u\times \bar{\phi}_{uu}+2u'^2v'\bar{\phi}_u\times \bar{\phi}_{uv}\\
+u'v'^2\bar{\phi_u}\times \bar{\phi}_{vv}+u'^2\dot{v}\bar{\phi}_v\times \bar{\phi}_{uu}+2u'v'^2\bar{\phi}_v\times \bar{\phi}_{uv}+v'^3\bar{\phi}_v\times \bar{\phi}_{vv}\Big],
\end{eqnarray*}
and hence
\begin{equation*} \bar{\gamma}(s)=\bar{\lambda}(s)\vec{\bar{t}}(s)+\frac{\bar{\mu}(s)}{\bar{k}(s)}\vec{\bar{b}}(s),\\
\end{equation*}
for some functions $\bar{\lambda}(s)$ and $\bar{\mu}(s)$. Therefore $\bar{\gamma}(s)$ is a rectifying curve on $\bar{S}$.
\end{proof}
\begin{note}
In the above theorem we see that the functions $\lambda(s)$ and $\bar{\lambda}(s)$ for the rectifying curves $\gamma(s)$ and $\bar{\gamma(s)}$ on $S$ and $\bar{S}$ respectively does not change while taking an isometry on $S$ to $\bar{S}$. Also $\frac{\bar{\mu}(s)}{\bar{k}(s)}=\frac{\mu(s)}{k(s)}$, i.e., $\mu(s)$ and $\bar{\mu}(s)$ for the rectifying curves $\gamma(s)$ and $\bar{\gamma(s)}$ respectively are related by the curvature functions $k(s)$ and $\bar{k}(s)$.
\end{note}
\begin{thm}\label{thm2}
Let $F$ be an isometry of two smooth surfaces $S$ and $\bar{S}$. For the rectifying curves $\gamma(s)$ and $\bar{\gamma}(s)$ on $S$ and $\bar{S}$ respectively the component of the position vector of the rectifying curve along normal to the surface is invariant under the isometry $F$, i.e., $\gamma(s)\cdot\vec{N}=\bar{\gamma}(s)\cdot\vec{\bar{N}}$.
\end{thm}
\begin{proof}
Since $F:S\rightarrow \bar{S}$ is an isometry and $\gamma(s)$, $\bar{\gamma}(s)$ are rectifying curves on $S$ and $\bar{S}$ respectively, the relations $(5)$, $(6)$ and $(*)$ hold.
Since $S$ and $\bar{S}$ are isometric, we have
\begin{equation}
E=\bar{E},\quad F=\bar{F},\quad G=\bar{G},
\end{equation}
and hence
\begin{equation*}
E=\bar{E}=\bar{\phi_u}\cdot\bar{\phi_u}=(F_*\phi_u)\cdot(F_*\phi_u),
\end{equation*}
\begin{equation}\label{7}
\text{i.e., }(F_*\phi_u)\cdot(F_*\phi_u)=\phi_u\cdot\phi_u.
\end{equation} 
Differentiating $(\ref{7})$ partially with respect to $u$ we get
\begin{equation*}
2\Big(\frac{\partial F_*}{\partial u}\phi_u+F_*\phi_{uu}\Big)\cdot(F_*\phi_u)=2\phi_{uu}\cdot\phi_u,
\end{equation*}
\begin{equation}\label{8}
\text{i.e., }\bar{\phi_{uu}}\cdot\bar{\phi_u}=\phi_{uu}\cdot\phi_u.
\end{equation} 
Again differentiating $(\ref{7})$ partially with respect to $v$ we get
\begin{equation*}
2\Big(\frac{\partial F_*}{\partial v}\phi_u+F_*\phi_{uv}\Big)\cdot(F_*\phi_u)=2\phi_{uv}\cdot\phi_u,
\end{equation*}
\begin{equation}\label{9}
\text{i.e., }\bar{\phi_{uv}}\cdot\bar{\phi_u}=\phi_{uv}\cdot\phi_u.
\end{equation}
Again
\begin{equation*}
G=\bar{G}=\bar{\phi_v}\cdot\bar{\phi_v}=(F_*\phi_v)\cdot(F_*\phi_v),
\end{equation*}
\begin{equation}\label{10}
\text{i.e., }(F_*\phi_v)\cdot(F_*\phi_v)=\phi_v\cdot\phi_v.
\end{equation}
Similarly differentiating $(\ref{10})$ partially with respect to $u$ and $v$ we get
\begin{equation}\label{11}
\bar{\phi_{uv}}\cdot\bar{\phi_v}=\phi_{uv}\cdot\phi_v,
\end{equation} 
and
\begin{equation}\label{12}
\bar{\phi_{vv}}\cdot\bar{\phi_v}=\phi_{vv}\cdot\phi_v.
\end{equation}
Again also
\begin{equation*}
\nonumber F=\bar{F}=\bar{\phi_u}\cdot\bar{\phi_v}=(F_*\phi_u)\cdot(F_*\phi_v),
\end{equation*}
\begin{equation}\label{13}
\text{i.e., }(F_*\phi_u)\cdot(F_*\phi_v)=\phi_u\cdot\phi_v.
\end{equation} 
Differentiating $(\ref{13})$ partially with respect to $u$ we get
\begin{equation*}
\Big(\frac{\partial F_*}{\partial u}\phi_u+F_*\phi_{uu}\Big)\cdot(F_*\phi_v)+(F_*\phi_u)\cdot\Big(\frac{\partial F_*}{\partial u}\phi_v+F_*\phi_{uv}\Big)=\phi_{uu}\cdot\phi_u+\phi_u\cdot\phi_{uv},
\end{equation*}
\begin{equation}\label{14}
\text{i.e., }\bar{\phi_{uu}}\cdot\bar{\phi_v}+\bar{\phi_u}\cdot\bar{\phi_{uv}}=\phi_{uu}\cdot\phi_v+\phi_u\cdot\phi_{uv}.
\end{equation} 
Using equation (\ref{9}) we can write equation $(\ref{14})$ as
\begin{equation}\label{15}
\bar{\phi_{uu}}\cdot\bar{\phi_v}=\phi_{uu}\cdot\phi_v. 
\end{equation}
Differentiating $(15)$ partially with respect to $v$ we get
\begin{equation*}
(\frac{\partial F_*}{\partial v}\phi_u+F_*\phi_{uv})\cdot(F_*\phi_v)+(F_*\phi_u)\cdot(\frac{\partial F_*}{\partial v}\phi_v+F_*\phi_{vv})=\phi_{uv}\cdot\phi_v+\phi_u\cdot\phi_{vv},
\end{equation*}
\begin{equation}\label{16}
\text{i.e., }\bar{\phi_{uv}}\cdot\bar{\phi_v}+\bar{\phi_u}\cdot\bar{\phi_{vv}}=\phi_{uv}\cdot\phi_v+\phi_u\cdot\phi_{vv}.
\end{equation} 
Using equation (\ref{11}) we can write equation $(\ref{16})$ as
\begin{equation}\label{17}
\bar{\phi_{vv}}\cdot\bar{\phi_u}=\phi_{vv}\cdot\phi_u. 
\end{equation}
Equation $(3)$ for the rectifying curve $\bar{\gamma}(s)$ can be written as
\begin{eqnarray}
\nonumber
\bar{\gamma}(s)\cdot \vec{\bar{N}}&=&\frac{\bar{\mu}(s)}{\bar{k}(s)}\Big[(u'v''-u''v)(\bar{E}\bar{G}-\bar{F}^2)+u'^3\{\bar{E}(\bar{\phi}_{uu}\cdot\bar{\phi}_v)-\bar{F}(\bar{\phi}_{uu}\cdot\bar{\phi}_u)\}+2u'^2v'\{\bar{E}(\bar{\phi}_{uv}\cdot\bar{\phi}_v)\\
\nonumber
&&-\bar{F}(\bar{\phi}_{uv}\cdot\bar{\phi}_u)\}+u'v'^2\{\bar{E}(\bar{\phi}_{vv}\cdot\bar{\phi}_v)-\bar{F}(\bar{\phi}_{vv}\cdot\bar{\phi}_u)\}+u'^2v'\{\bar{F}(\bar{\phi}_{uu}\cdot\bar{\phi}_v)-\bar{G} (\bar{\phi}_{uu}\cdot\bar{\phi}_u)\}\\
\nonumber
&&+2u'v'^2\{\bar{F}(\bar{\phi}_{uv}\cdot\bar{\phi}_v)-\bar{G}(\bar{\phi}_{uv}\cdot\bar{\phi}_u)\}+v'^3\{\bar{F}(\bar{\phi}_{vv}\cdot\bar{\phi}_v)-\bar{G}(\bar{\phi}_{vv}\cdot\bar{\phi}_u)\}.
\end{eqnarray}
By virtue of $(10)$ (\ref{8}), (\ref{9}), (\ref{11}), (\ref{12}), (\ref{15}) and (\ref{17}), the last relation yields
\begin{equation*}
\bar{\gamma}(s)\cdot\vec{\bar{N}}=\gamma(s)\cdot\vec{N}.
\end{equation*}
Therefore the component of a rectifying curve $\gamma(s)$ along normal to the surface $S$ is invariant under the rectifying curve preserving isomerty of surfaces.
\end{proof}
\section{acknowledgment}
We are immensely grateful to Professor Dr. Bang-Yen Chen, Michigan State University, for his valuable comments. The second author greatly acknowledges to The University Grants Commission, Government of India for the award of Junior Research Fellow.

\end{document}